\documentclass{amsart}
\usepackage[utf8]{inputenc}
\usepackage{amsaddr}
\usepackage{hyperref}

\usepackage{setspace}
\usepackage{amsmath}
\usepackage{linegoal}
\usepackage{amsthm}
\usepackage{amssymb}
\usepackage{mathtools}
\usepackage{amsfonts}
\usepackage{enumitem}
\usepackage{wasysym}
\usepackage{subcaption} 
\usepackage{graphicx}
\usepackage{float}
\usepackage{tikz}
\usepackage{mathrsfs}
\usetikzlibrary{external}
\usetikzlibrary{arrows.meta}
\usetikzlibrary{shapes.geometric}
\usepackage{xcolor-material}
\usepackage[ruled,vlined,linesnumbered]{algorithm2e}
\usepackage{algorithmic}
\usepackage[backend=biber,style=alphabetic,]{biblatex}
\usepackage{listings}
\newtheorem{theorem}{Theorem}[section]
\newtheorem{remark}[theorem]{Remark}
\newtheorem{lemma}[theorem]{Lemma}

\newtheorem{corollary}[theorem]{Corollary}

\theoremstyle{definition}
\newtheorem{definition}[theorem]{Definition}
\newtheorem{example}[theorem]{Example}
\newtheorem{question}[theorem]{Question}
\newtheorem{algo}[theorem]{Algorithm}

\title[Long and Short Periodic Trajectories in the Pentagon]{Long and Short Periodic Billiard Trajectories \\ in the Regular Pentagon}
\author{Samuel Everett, Vanessa Lin, Aidan Mager}

\begin{document}

\begin{abstract}
    In any periodic direction on the regular pentagon billiard table, there exists two combinatorially different billiard paths, with one longer than the other. For each periodic direction, McMullen asked if one could determine whether the periodic trajectory through a given point is long, short, or a saddle connection. In this paper we present an algorithm resolving this question for trajectories emanating from the midpoints of the pentagon.
\end{abstract}

\maketitle

\section{Introduction}
The theory of polygonal billiards concerns the uniform motion of a point mass (billiard ball) in a polygonal plane domain (billiard table). We define collisions with the boundary to be elastic: the angle of incidence equals the angle of reflection. It is natural to consider the behavior of periodic billiard orbits in polygons. Indeed, a long-standing open problem in polygonal billiards is whether every polygon contains a periodic billiard orbit (see \cites{gutkin, gutkint} for surveys).  Much work has been devoted to this particular problem, leading to significant progress (see e.g., \cites{masur, schwartz, galperin}), although many questions remain unanswered.

\begin{figure}[h]
    \begin{subfigure}{0.5\textwidth}
        \centering
        \includegraphics[scale=.5]{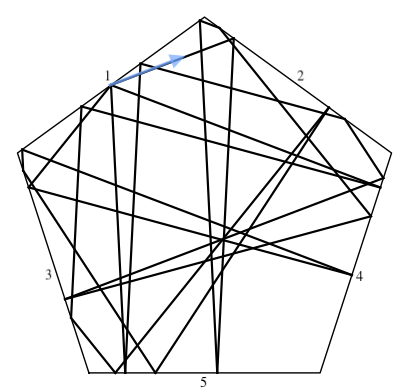}
    \end{subfigure}%
    \begin{subfigure}{0.5\textwidth}
        \centering
        \includegraphics[scale=.5]{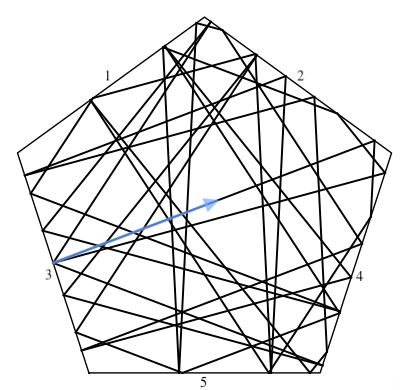}
    \end{subfigure}
    \caption{A short trajectory (left) and a long trajectory (right) in the regular pentagon billiard table beginning in the periodic direction 132.}
    \label{fig:pent_trajs}
\end{figure}

One fruitful line of study has come from billiards in polygons whose interior angles are rational multiples of $\pi$ (rational polygons), and their correspondence with straight line flow on translation surfaces via \textit{unfolding} (see \cites{wrightsurv, MasurTabach, billiards} for review). This correspondence has lead to much development in both billiards and translation surfaces. For instance, Masur in \cite{masur} used this correspondence to show that every rational polygon admits countably many periodic billiard trajectories.

It is also interesting to characterize the behavior of periodic billiard orbits in particular classes of polygons (see e.g. \cites{dft, dl}).  In \cite{dl}, Davis and Leli\`evre study the rich behavior of periodic billiard trajectories in the regular pentagon by using the correspondence of the regular pentagon billiard table with the double pentagon and golden L translation surfaces. Their work demonstrated how any periodic billiard trajectory in the regular pentagon is either long, short, or hits a corner. To illustrate, a short trajectory and a long trajectory for a particular periodic direction is shown in Figure \ref{fig:pent_trajs}. As a consequence, McMullen \cite{mcmullenQ} asked if given a periodic direction and midpoint of an edge, is there a way to determine if the trajectory is long, short, or hits a corner without having to explicitly draw the trajectory? The following theorem proved in this paper resolves this question.

\begin{theorem} \label{thm2}
Let $\vec{v}$ be a periodic billiard direction in the regular pentagon billiard table. For any choice of midpoint on a side, Algorithm \ref{alg:MainAlgorithm} gives a process to determine if the billiard trajectory in the direction of $\vec{v}$ is long, short, or hits a corner point.
\end{theorem}

The proof of Theorem \ref{thm2} is an algorithm which translates the problem of billiards on the pentagon to straight-line flow on the golden L translation surface, and determines if the associated trajectory is in a long or short cylinder.  

This paper is organized as follows.  In \S \ref{sec1} we provide background definitions and theory necessary for stating the proof of Theorem \ref{thm2}, which we give in \S \ref{sec2}. Finally, a corollary that often reduces the steps needed in the algorithm is explored in \S \ref{sec3}.

\section{Preliminaries}\label{sec1}

\subsection{Translation surfaces and Veech groups}
Billiards are characterized by their elastic reflection off the boundary of the table according to the mirror law of reflection. This allows us to relate billiard dynamics to linear flow on translation surfaces.
In particular, reflecting across the edges of the table until each edge is paired with an opposite parallel edge gives rise to an equivalent representation on a translation surface.
These surfaces are typically easier to study because it allows us to restrict our attention to a single fixed direction.

We begin by giving the definition of translation surfaces and their corresponding properties.  For a complete review of translation surfaces see \cites{wrightsurv, MasurTabach}.  

\begin{definition}
A \emph{translation surface}, denoted $(X, \omega)$, is a disjoint union of polygons in $\mathbb{C}$ with opposite, parallel sides identified by translation.  In particular, a translation surface is embedded in $\mathbb{C}$ with embedding fixed only up to translation.  In the polygonal representation, we consider translation surfaces to be equivalent if we can achieve one from the other through cut-and-paste operations.
\end{definition}

\begin{remark}
The notation $(X, \omega)$ of a translation surface comes from an equivalent definition of a translation surface.  Namely, a translation surface $(X, \omega)$ is a nonzero Abelian differential $\omega$ on a Riemann surface $X$.
\end{remark}

The vertices of the polygons defining a translation surface are called \textit{cone points}, and the segments joining pairs of cone points without any cone points in the interior are called \textit{saddle connections}. Trajectories that hit the corner of a given polygonal billiard table correspond to saddle connections in the translation surface obtained via unfolding said table.

A \textit{cylinder} of a translation surface is a maximal family of periodic trajectories (closed geodesics) in a periodic direction, with boundaries given by saddle connections.  See Figure \ref{DPSurf} for visual demonstration of cylinders.

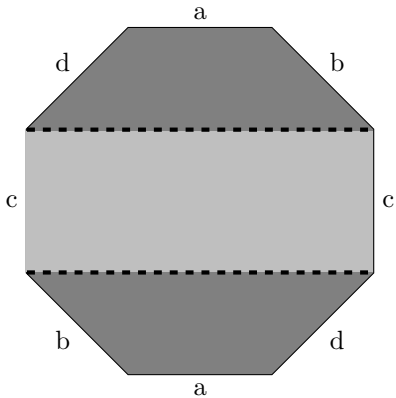
\begin{figure}[h]
    \centering
    \begin{tikzpicture}[scale=1]
        \node[regular polygon,regular polygon sides=8, minimum size=5cm, draw, fill=gray] at (0,0) {};
        \node[rectangle, fill=lightgray, minimum width = 4.6cm, 
        minimum height = 1.87cm] (r) at (0,0) {};
        \draw[ultra thick, dashed] (-2.3, 0.95)--(2.3, 0.95); 
        \draw[ultra thick, dashed] (-2.3, -0.95)--(2.3, -0.95);
        \draw (2.3, 0) node[anchor=west] {c};
        \draw (-2.3, 0) node[anchor=east] {c};
        \draw (-1.6, 1.6) node[anchor=south east] {d};
        \draw (1.6, 1.6) node[anchor=south west] {b};
        \draw (1.6, -1.6) node[anchor=north west] {d};
        \draw (-1.6, -1.6) node[anchor=north east] {b};
        \draw (0, 2.3) node[anchor=south] {a};
        \draw (0, -2.3) node[anchor=north] {a};
    \end{tikzpicture}
    \caption{The cylinder decomposition of the regular octagon translation surface in the horizontal direction. A longer cylinder (dark gray) and a shorter cylinder (light gray) are separated by saddle connections (black dashes).}
    \label{DPSurf}
\end{figure}

Translation surfaces admit an action by $\operatorname{SL}_2\mathbb{R}$.  If $g \in \operatorname{SL}_2\mathbb{R}$ and $(X, \omega)$ is a translation surface given as a collection of polygons, then $g(X, \omega)$ is the translation surface obtained by acting linearly by $g$ on the polygons determining $(X, \omega)$.

\begin{definition}
We denote the stabilizer of $(X, \omega)$ under the action of $\operatorname{SL}_2\mathbb{R}$ by $\operatorname{SL}(X, \omega)$.  The \emph{Veech Group} of $(X, \omega)$ is the image of $\operatorname{SL}(X, \omega)$ in $\operatorname{PSL}_2\mathbb{R}$.  If $\operatorname{SL}(X, \omega)$ is a lattice, then $(X, \omega)$ is a \emph{Veech surface}.  
\end{definition}

\subsection{Regular pentagon billiard table and double pentagon surface} \label{sec:unfoldingToDoublePent}
Straight line flow on a translation surface corresponds to billiard flow in a polygon through \textit{unfolding}. That is, instead of reflecting the billiard off a side of a billiard polygon, reflect the polygon on a side and unfold the trajectory to a straight line, as shown in Figure \ref{unfolding}.  Reflecting polygons in such a manner can generate translation surfaces, and hence we obtain a correspondence between billiard flow in a polygon and straight line flow on a translation surface.  For a detailed description of unfolding a billiard into straight line flow on a translation surface, see \cite[\S 1.3]{MasurTabach}.

The regular pentagon table can be unfolded to the \textit{necklace}, a 5-fold cover of the double pentagon translation surface.  As a consequence, we can examine straight-line flow on the double pentagon translation surface to understand billiard dynamics in the regular pentagon.

\begin{figure}
    \centering
    \includegraphics[scale=.23]{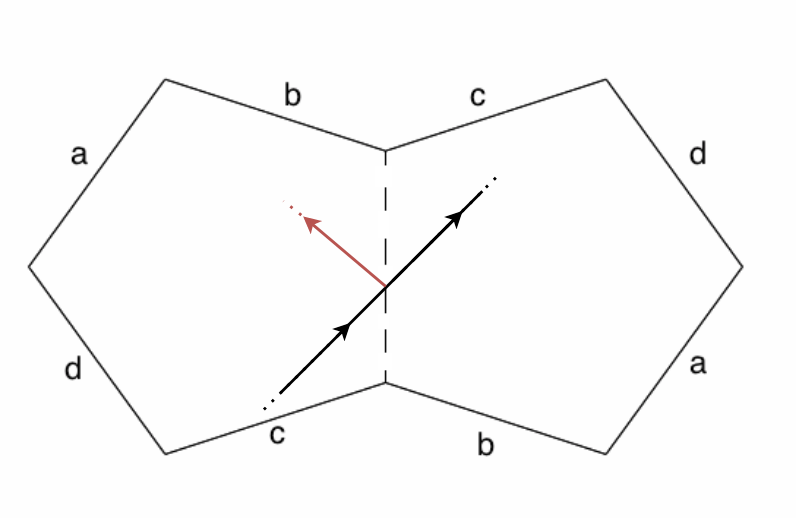}
    \caption{An unfolding of a billiard trajectory in the regular pentagon table to a flow in the double pentagon surface. The red trajectory shows the billiard trajectory reflected off the edge.}
    \label{unfolding}
\end{figure}

\subsection{Going from the double pentagon to the golden L}
\begin{figure}[h]
    \centering
    \begin{tikzpicture}[scale=0.86]
        \draw (0,0) -- (1.61803, 0) -- (1.61803^2, 0) -- (1.61803^2, 1.61803) -- (1.61803, 1.61803) -- (1.61803, 1.61803^2) -- (0, 1.61803^2) -- (0, 1.61803) -- (0,0);
        \draw (-0.1, 1.61803) -- (0.1, 1.61803);
        \draw (1.61803, -0.1) -- (1.61803, 0.1);
        \draw (0, 1.61803/2) node[anchor=east] {b};
        \draw (1.61803^2, 1.61803/2) node[anchor=west] {b};
        \draw (0, 1.61803 + 1/2) node[anchor=east] {a};
        \draw (1.61803, 1.61803 + 1/2) node[anchor=west] {a};
        \draw (1.61803/2, 1.61803^2) node[anchor=south] {c};
        \draw (1.61803/2, 0) node[anchor=north] {c};
        \draw (1.61803 + 1/2, 0) node[anchor=north] {d};
        \draw (1.61803 + 1/2, 1.61803) node[anchor=south] {d};
    \end{tikzpicture}  
    \caption{The golden L with edge identifications labeled.}
    \label{fig:goldenL}
\end{figure}
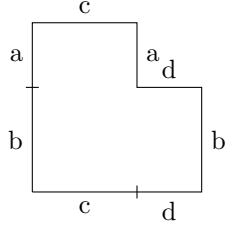
Recall the golden ratio $\phi$, defined to be
$$\phi = \frac{1+\sqrt{5}}{2} \approx 1.618$$
and is obtained as one of the solutions of $x^2 - x - 1 = 0$. Hence, it satisfies a useful identity: $\phi^2 = \phi +1$. We construct the \textit{golden L} by taking a $\phi \times \phi$ square, gluing two $1\times \phi$ rectangles to adjacent sides, and identifying opposite and parallel sides (see Figure \ref{fig:goldenL}). We construct the \textit{double pentagon} by gluing two pentagons together and identifying opposite and parallel sides, as in Figure \ref{unfolding}. Note this constructions is equivalent to the double pentagon obtained from the unfolding in \S \ref{sec:unfoldingToDoublePent} (see \cite[\S 2.1]{dl} for further information on the double pentagon). As in \cite[Definition 2.2]{dl}, let $$P = \begin{pmatrix} 1 & \cos{\pi/5} \\ 0 & \sin{\pi/5} \end{pmatrix}.$$

\begin{lemma} \cite[Lemma 2.3]{dl} \label{lem:LongShortConnec}
The matrix $P$ takes the golden L surface to the double pentagon surface, and its inverse $P^{-1}$ takes the double pentagon to the golden L. In particular, they take long cylinder vectors to long cylinder vectors, and the same for short cylinder vectors.
\end{lemma}

\subsection{Periodic trajectories in the regular pentagon, the double pentagon, and the golden L} \label{sec:PerDirections}
We begin with the \text{golden L}, pictured in Figure \ref{fig:goldenL}. The following matrices generate the Veech group of the Golden L \cite[Lemma 2.6]{dl}:
\begin{center}
    $\sigma _0 = \begin{pmatrix} 1 & \phi \\ 0 & 1 \end{pmatrix}$, 
    $\sigma _1 = \begin{pmatrix} \phi & \phi \\ 1 & \phi \end{pmatrix}$, 
    $\sigma _2 = \begin{pmatrix} \phi & 1 \\ \phi & \phi \end{pmatrix}$, 
    $\sigma _3 = \begin{pmatrix} 1 & 0 \\ \phi & 1 \end{pmatrix}$
\end{center}
The images of each $\sigma_i$ acting on the golden L are pictured in Figures \ref{sigma0app} - \ref{sigma3app}. Since these are elements of the Veech group of the golden L, we know that they send the golden L to itself, up to a cut and paste. See Figure \ref{cutandpaste} for a visual demonstration of how the cut and paste operation works on the golden L after applying $\sigma_0$.  Note the cut and paste operation respects side identifications and the cut lines may become new side identifications. \par
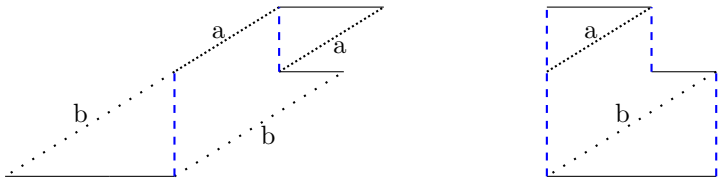
\begin{figure}[h]
    \begin{subfigure}[h]{0.45\textwidth}
        \centering
        \begin{tikzpicture}[scale=0.86]
            \draw (0,0) -- (1.61803, 0);
            \draw (1.61803, 0) -- (1.61803+1, 0);
            \draw[thick, loosely dotted] (1.61803+1, 0) -- (2*1.61803+2, 1.61803);
            \draw[thick, loosely dotted] (1.61803+1, 1.61803) -- (0,0);
            \draw (2*1.61803+2, 1.61803) -- (2*1.61803+1, 1.61803);
            \draw[thick, densely dotted] (2*1.61803+1, 1.61803) -- (3*1.61803+1, 1.61803+1);
            \draw[thick, densely dotted] (2*1.61803+1, 1.61803+1) -- (1.61803+1, 1.61803);
            \draw (3*1.61803+1, 1.61803+1) -- (2*1.61803+1, 1.61803+1);
            \draw[blue,thick,dashed] (1.61803+1, 0) -- (1.61803+1, 1.61803);
            \draw[blue,thick,dashed] (2*1.61803+1, 1.61803) -- (2*1.61803+1, 1.61803+1);
            \draw (2.5*1.61803 + 1, 1.61803 +1/2) node[anchor=north west, shift={(-1mm,1mm)}] {a};
            \draw (1.5*1.61803 + 1, 1.61803 +1/2) node[anchor=south east, shift={(1mm,-1mm)}] {a};
            \draw (1.5*1.61803 + 3/2, 1.61803/2) node[anchor=north west, shift={(-1mm,1mm)}] {b};
            \draw (1.61803/2 + 1/2, 1.61803/2) node[anchor=south east, shift={(1mm,-1mm)}] {b};
        \end{tikzpicture}  
    \end{subfigure}
    \begin{subfigure}[h]{0.45\textwidth}
        \centering
        \begin{tikzpicture}[scale=0.86]\textbf{}
            \draw[thick, loosely dotted] (1.61803+1, 1.61803) -- (0,0);
            \draw (0,0) -- (1.61803, 0) -- (1.61803+1, 0);
            \draw (1.61803+1, 1.61803) -- (1.61803, 1.61803);
            \draw[thick, densely dotted] (0, 1.61803) -- (1.61803, 1.61803+1);
            \draw (1.61803, 1.61803+1) -- (0, 1.61803+1);
            \draw[blue,thick,dashed] (1.61803+1, 0) -- (1.61803+1, 1.61803);
            \draw[blue,thick,dashed] (0, 0) -- (0, 1.61803+1);
            \draw[blue,thick,dashed] (1.61803, 1.61803) -- (1.61803, 1.61803+1);
            \draw (1.61803/2, 1.61803 +1/2) node[anchor=south east, shift={(1mm,-1mm)}] {a};
            \draw (1.61803/2 + 1/2, 1.61803/2) node[anchor=south east, shift={(1mm,-1mm)}] {b};
        \end{tikzpicture}  
    \end{subfigure}
    \caption{$\sigma_0$ acting on the golden L before and after cutting and pasting along the blue dotted lines.}
    \label{cutandpaste}
\end{figure}
Next, consider a trajectory within the double pentagon. We say a trajectory in a translation surface is a \textit{periodic trajectory} if it forms a closed curve. In other words, the trajectory eventually comes back to where it started and repeats. It was shown by Davis-Leli\`evre in \cite{dl} that periodic trajectories in the golden L are precisely those with slope in $\mathbb{Q}[\sqrt{5}]$.
The following theorem from \cite{dl} gives us a way to express any periodic direction as a product of finitely many $\sigma_i$'s and the horizontal vector $(\begin{smallmatrix}
      1\\
      0
    \end{smallmatrix}\big)$:

\begin{theorem} \cite[Theorem 2.22]{dl} \label{thm:PerDirec}
    Corresponding to any periodic direction vector $\vec{v}$ in the first quadrant on the golden L is a unique sequence $a_v=(k_1, k_2, \dots, k_n)$ of sectors such that $\vec{v} = \ell_v \sigma_{k_n}\dots \sigma_{k_1}\big(\begin{smallmatrix}
      1\\
      0
    \end{smallmatrix}\big)$ for some length $\ell_v$.
\end{theorem}
We refer to this sequence of integers $k_1 k_2 \dots k_n$, $k_j \in (0,1,2,3)$, as a \textit{tree word}. See Example \ref{ex:treeword} below for an example of calculating the periodic direction vector from a tree word.

\begin{example}
\label{ex:treeword}
    The vector $\vec{v}$ in the direction $132$ is given by
    \begin{align*}
        v &= \sigma_2 \cdot \sigma_3 \cdot \sigma_1 \cdot \begin{pmatrix} 1\\ 0 \end{pmatrix} \\
        &= \begin{pmatrix} \phi & 1 \\ \phi & \phi \end{pmatrix} \cdot \begin{pmatrix} 1 & 0 \\ \phi & 1 \end{pmatrix} \cdot \begin{pmatrix} \phi & \phi \\ 1 & \phi \end{pmatrix} \cdot \begin{pmatrix} 1\\ 0 \end{pmatrix} \\
        &= \begin{pmatrix} 2\phi^2 + 1 \\ 2\phi^2 + 2\phi \end{pmatrix} 
        = \begin{pmatrix} 2\phi + 3 \\ 4\phi + 2 \end{pmatrix} \tag{Recall $\phi^2 = \phi +1$}.
    \end{align*}
\end{example}
Using the following corollary, we can jump between periodic directions in the double pentagon to periodic directions in the golden L:
\begin{corollary} \cite[Corollary 2.4]{dl} \label{cor:PinvIsPer}
    A direction $\vec{v}$ is periodic on the double pentagon if and only if the direction $P^{-1}\vec{v}$ is periodic on the golden L.
\end{corollary}

\subsection{Weierstrass points and the Veech group}
\label{weierstrass}
Next, we inscribe a pentagon within the golden L as pictured in Figure \ref{fig:goldLWP}, and label the midpoints of the sides as shown. These midpoints coincide with the \emph{Weierstrass points}, which are also the midpoints of the edges in the double pentagon. Moreover, these are the labeled points in the regular pentagon as in Figure \ref{fig:goldLWP}.

\begin{figure}[h]
    \centering
    \begin{subfigure}{0.4\textwidth}
        \centering
        \begin{tikzpicture}[scale=0.86]
            \draw (0,0) -- (1.61803, 0) -- (1.61803^2, 0) -- (1.61803^2, 1.61803) -- (1.61803, 1.61803) -- (1.61803, 1.61803^2) -- (0, 1.61803^2) -- (0, 1.61803) -- (0,0);
            \draw [red] (1.61803, 0) -- (1.61803^2, 0) -- (1.61803, 1.61803) -- (0, 1.61803^2) -- (0, 1.61803) -- (1.61803, 0);
            \filldraw[red] (0, 1.61803+0.5) circle (2pt) node[anchor=west] {1};
            \filldraw[red] (1.61803/2, 1.61803+0.5) circle (2pt) node[anchor=west] {2};
            \filldraw[red] (1.61803/2, 1.61803/2) circle (2pt) node[anchor=west] {3};
            \filldraw[red] (1.61803+0.5, 1.61803/2) circle (2pt) node[anchor=west] {4};
            \filldraw[red] (1.61803+0.5, 0) circle (2pt) node[anchor=south] {5};
        \end{tikzpicture}  
    \end{subfigure}
    \begin{subfigure}{0.4\textwidth}
        \centering
        \begin{tikzpicture}[scale=0.86]
            \node[regular polygon,regular polygon sides=5, minimum size=2.6cm, draw] at (0,0) {};
            \filldraw[red] (-0.77, 0.94) circle (2pt) node[anchor=west] {1};
            \filldraw[red] (0.77, 0.94) circle (2pt) node[anchor=east] {2};
            \filldraw[red] (1.16, -0.4) circle (2pt) node[anchor=east] {4};
            \filldraw[red] (-1.16, -0.4) circle (2pt) node[anchor=west] {3};
            \filldraw[red] (0, -1.2) circle (2pt) node[anchor=south] {5};
        \end{tikzpicture}
    \end{subfigure}
    \caption{Left: The golden L with the inscribed pentagon and Weierstrass points labeled. Right: The corresponding points labeled on the regular pentagon.}
    \label{fig:goldLWP}
\end{figure}
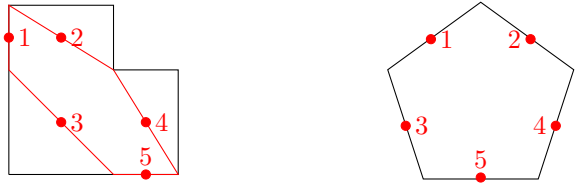

We say a point of a Veech surface is \textit{periodic} if it has finite orbit under action of the Veech group.  The Weierstrass points are the only periodic points of the golden L under action of the Veech group (See \cite[\S 2]{wrightsurv} for review of periodic points on Veech surfaces and L tables).  As a consequence, the action of the Veech group elements over the golden L permutes the points, as seen in Figures \ref{sigma0app} - \ref{sigma3app}.  

\begin{figure}[t]
    \centering
    \begin{subfigure}{0.4\textwidth}
        \centering
        \begin{tikzpicture}[scale=0.86]
            \draw (0,0) -- (1.61803, 0) -- (1.61803+1, 0) -- (2*1.61803+2, 1.61803) -- (2*1.61803+1, 1.61803) -- (3*1.61803+1, 1.61803+1) -- (2*1.61803+1, 1.61803+1) -- (1.61803+1, 1.61803) -- (0,0);
            \draw[red] (1.61803, 0) -- (1.61803+1, 0) -- (2*1.61803+1, 1.61803) -- (2*1.61803+1, 1.61803+1) -- (1.61803+1, 1.61803) -- (1.61803, 0);
            \filldraw[red] (1.61803^2 + 1.61803/2, 1.61803+0.5) circle (2pt) node[anchor=east] {1};
            \filldraw[red] (2*1.61803+1, 1.61803+0.5) circle (2pt) node[anchor=west] {2};
            \filldraw[red] (1.61803+0.5, 1.61803/2) circle (2pt) node[anchor=west] {3};
            \filldraw[red] (1.61803+0.5+1.61803^2/2, 1.61803/2) circle (2pt) node[anchor=west] {4};
            \filldraw[red] (1.61803+0.5, 0) circle (2pt) node[anchor=south] {5};
        \end{tikzpicture}  
    \end{subfigure}%
    \begin{subfigure}{0.4\textwidth}
        \centering
        \begin{tikzpicture}[scale=0.86]
            \draw (1.61803+1, 1.61803) -- (0,0) -- (1.61803, 0) -- (1.61803+1, 0);
            \draw (0, 0) -- (1.61803+1, 1.61803) -- (1.61803, 1.61803);
            \draw (0, 1.61803) -- (1.61803, 1.61803+1) -- (0, 1.61803+1) -- (0, 1.61803);
            \draw (1.61803+1, 0) -- (1.61803+1, 1.61803);
            \draw (0, 0) -- (0, 1.61803+1);
            \draw (1.61803, 1.61803) -- (1.61803, 1.61803+1);
            \draw[red] (1.61803+1, 1.61803) -- (1.61803, 0) -- (1.61803+1, 0);
            \draw[red] (0, 0) -- (1.61803, 1.61803); 
            \draw[red] (0, 1.61803+1) -- (0, 1.61803) -- (1.61803, 1.61803+1);
            \filldraw[red] (1.61803/2, 1.61803+0.5) circle (2pt) node[anchor=west] {1};
            \filldraw[red] (0, 1.61803+0.5) circle (2pt) node[anchor=west] {2};
            \filldraw[red] (1.61803+0.5, 1.61803/2) circle (2pt) node[anchor=west] {3};
            \filldraw[red] (1.61803/2, 1.61803/2) circle (2pt) node[anchor=west] {4};
            \filldraw[red] (1.61803+0.5, 0) circle (2pt) node[anchor=south] {5};
        \end{tikzpicture}  
    \end{subfigure}
    \caption{Image under $\sigma_0$ before and after cut and paste.}
    \label{sigma0app}
\end{figure}
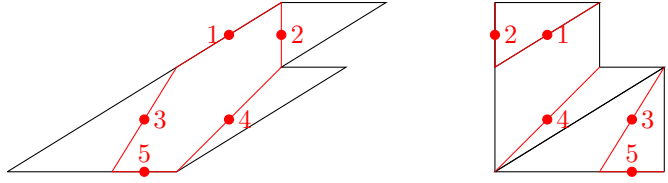
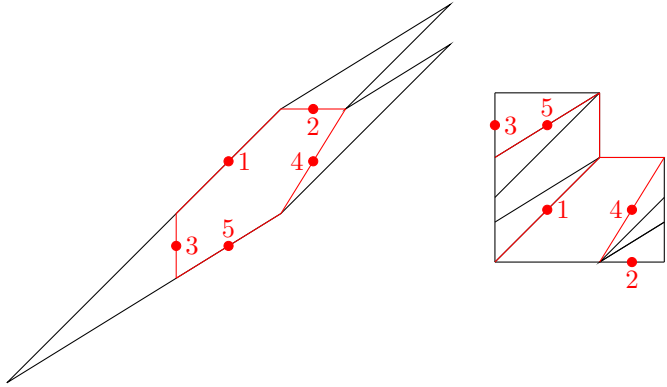
\begin{figure}[t]
    \begin{subfigure}{0.4\textwidth}
        \centering
        \begin{tikzpicture}[scale=0.86]
            \draw (0,0) -- (1.61803+1, 1.61803) -- (2*1.61803+1, 1.61803+1) -- (3*1.61803+2, 2*1.61803+2) -- (2*1.61803+2, 2*1.61803+1) -- (3*1.61803+2, 3*1.61803+1) -- (2*1.61803+1, 2*1.61803+1) -- (1.61803+1, 1.61803+1) -- (0,0);
            \draw [red] (1.61803+1, 1.61803) -- (2*1.61803+1, 1.61803+1) -- (2*1.61803+2, 2*1.61803+1) -- (2*1.61803+1, 2*1.61803+1) -- (1.61803+1, 1.61803+1) -- (1.61803+1, 1.61803);
            \filldraw[red] (1.5*1.61803+1, 1.5*1.61803+1) circle (2pt) node[anchor=west] {1};
            \filldraw[red] (2*1.61803+1.5, 2*1.61803+1) circle (2pt) node[anchor=north] {2};
            \filldraw[red] (1.61803+1, 1.61803+0.5) circle (2pt) node[anchor=west] {3};
            \filldraw[red] (2*1.61803+1.5, 1.5*1.61803+1) circle (2pt) node[anchor=east] {4};
            \filldraw[red] (1.5*1.61803+1, 1.61803+0.5) circle (2pt) node[anchor=south] {5};
        \end{tikzpicture}  
    \end{subfigure}%
    \begin{subfigure}{0.4\textwidth}
        \centering
        \begin{tikzpicture}[scale=0.86]
            \draw (1.61803, 1.61803+1) -- (0, 1); 
            \draw (1.61803+1, 1) -- (1.61803,0) -- (1.61803+1, 1/1.61803); 
            \draw (0, 1/1.61803) -- (1.61803, 1.61803); 
            \draw (0, 1.61803) -- (1.61803, 1.61803+1); 
            \draw (1.61803+1, 1/1.61803) -- (1.61803, 0); 
            \draw (1.61803, 1.61803) -- (0, 0); 
            \draw (0, 1.61803+1) -- (1.61803, 1.61803+1);
            \draw (1.61803+1, 1.61803) -- (1.61803, 1.61803);
            \draw (1.61803+1, 0) -- (0, 0);
            \draw (1.61803, 1.61803) -- (1.61803, 1.61803+1);
            \draw (0, 0) -- (0, 1.61803+1);
            \draw (1.61803+1, 0) -- (1.61803+1, 1.61803);
            \draw[red] (1.61803, 1.61803+1) -- (1.61803, 1.61803);
            \draw[red] (0, 1.61803) -- (1.61803, 1.61803+1);
            \draw[red] (1.61803, 0) -- (1.61803+1, 1.61803) -- (1.61803, 1.61803) -- (0, 0);
            \filldraw[red] (0.5*1.61803, 0.5*1.61803) circle (2pt) node[anchor=west] {1};
            \filldraw[red] (1.61803+0.5, 0) circle (2pt) node[anchor=north] {2};
            \filldraw[red] (0, 1.61803+0.5) circle (2pt) node[anchor=west] {3};
            \filldraw[red] (1.61803+0.5, 0.5*1.61803) circle (2pt) node[anchor=east] {4};
            \filldraw[red] (0.5*1.61803, 1.61803+0.5) circle (2pt) node[anchor=south] {5};
        \end{tikzpicture}  
    \end{subfigure}
    \caption{Image under $\sigma_1$ before and after cut and paste.}
    \label{sigma1app}
\end{figure}
\begin{figure}[t]
    \begin{subfigure}{0.4\textwidth}
        \centering
        \begin{tikzpicture}[scale=0.86]
            \draw (0, 0) -- (1.61803+1, 1.61803+1) -- (2*1.61803+1, 2*1.61803+1) -- (3*1.61803+1, 3*1.61803+2) -- (2*1.61803+1, 2*1.61803+2) -- (2*1.61803+2, 3*1.61803+2) -- (1.61803+1, 2*1.61803+1) -- (1.61803, 1.61803+1) -- (0, 0);
            \draw [red] (1.61803+1, 1.61803+1) -- (2*1.61803+1, 2*1.61803+1) -- (2*1.61803+1, 2*1.61803+2) -- (1.61803+1, 2*1.61803+1) -- (1.61803, 1.61803+1) -- (1.61803+1, 1.61803+1);
            \filldraw[red] (1.61803+0.5, 1.5*1.61803+1) circle (2pt) node[anchor=west] {1};
            \filldraw[red] (1.5*1.61803+1, 2*1.61803+1.5) circle (2pt) node[anchor=north] {2};
            \filldraw[red] (1.61803+0.5, 1.61803+1) circle (2pt) node[anchor=south] {3};
            \filldraw[red] (2*1.61803+1, 2*1.61803+1.5) circle (2pt) node[anchor=east] {4};
            \filldraw[red] (1.5*1.61803+1, 1.5*1.61803+1) circle (2pt) node[anchor=south] {5};
        \end{tikzpicture}  
    \end{subfigure}%
    \begin{subfigure}{0.4\textwidth}
        \centering
        \begin{tikzpicture}[scale=0.86]
            \draw (1.61803+1, 1.61803) -- (1.61803, 0); 
            \draw (1.61803, 1.61803+1) -- (1, 1.61803) -- (2, 1.61803+1);
            \draw (2, 0) -- (1.61803+2, 1.61803); 
            \draw (1, 0) -- (1.61803+1, 1.61803);
            \draw (1, 1.61803) -- (1.61803, 1.61803+1);
            \draw (1.61803, 0) -- (1.61803+1, 1.61803) -- (1, 0); 
            \draw (1.61803+1, 0) -- (1.61803+2, 1.61803) -- (2, 0);
            \draw (2, 1.61803+1) -- (1, 1.61803);
            \draw (1.61803+2, 1.61803) -- (1.61803+1, 0); 
            \draw (1, 0) -- (3, 0); 
            \draw (1.61803+2, 1.61803) -- (1.61803+1, 1.61803);
            \draw (1, 1.61803+1) -- (1.61803+1, 1.61803+1);
            \draw (1,0) -- (1, 1.61803+1);
            \draw (1.61803+2, 0) -- (1.61803+2, 1.61803);
            \draw[red] (1, 0) -- (1.61803+1, 1.61803) -- (1.61803+1, 1.61803+1) -- (1, 1.61803);
            \draw[red] (1.61803+2, 1.61803) -- (1.61803+1, 0) -- (1.61803+2, 0); 
            \filldraw[red] (1.61803+1.5, 0.5*1.61803) circle (2pt) node[anchor=west] {1};
            \filldraw[red] (0.5*1.61803+1, 1.61803+0.5) circle (2pt) node[anchor=north] {2};
            \filldraw[red] (1.61803+1.5, 0) circle (2pt) node[anchor=north] {3};
            \filldraw[red] (1, 1.61803+0.5) circle (2pt) node[anchor=east] {4};
            \filldraw[red] (0.5*1.61803+1, 0.5*1.61803) circle (2pt) node[anchor=south] {5};
        \end{tikzpicture}  
    \end{subfigure}
    \caption{Image under $\sigma_2$ before and after cut and paste.}
    \label{sigma2app}
\end{figure}
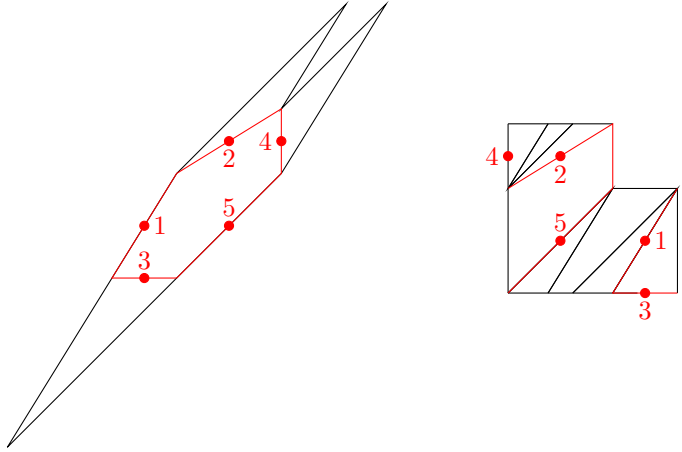
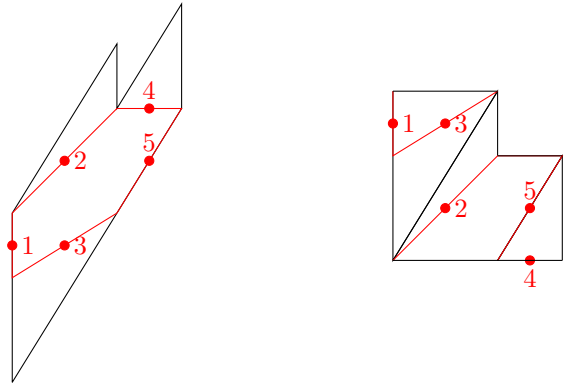
\begin{figure}[t]
    \begin{subfigure}{0.4\textwidth}
        \centering
        \begin{tikzpicture}[scale=0.86]
            \draw (0,0) -- (1.61803, 1.61803+1) -- (1.61803+1, 2*1.61803+1) -- (1.61803+1, 3*1.61803+1) -- (1.61803, 2*1.61803+1) -- (1.61803, 2*1.61803+2) -- (0, 1.61803+1) -- (0, 1.61803) -- (0,0);
            \draw [red] (1.61803, 1.61803+1) -- (1.61803+1, 2*1.61803+1) -- (1.61803, 2*1.61803+1) -- (0, 1.61803+1) -- (0, 1.61803) -- (1.61803, 1.61803+1);
            \filldraw[red] (0, 1.61803+0.5) circle (2pt) node[anchor=west] {1};
            \filldraw[red] (1.61803/2, 1.5*1.61803+1) circle (2pt) node[anchor=west] {2};
            \filldraw[red] (1.61803/2, 1.61803+0.5) circle (2pt) node[anchor=west] {3};
            \filldraw[red] (1.61803+0.5, 2*1.61803+1) circle (2pt) node[anchor=south] {4};
            \filldraw[red] (1.61803+0.5, 1.5*1.61803+1) circle (2pt) node[anchor=south] {5};
        \end{tikzpicture}  
    \end{subfigure}%
    \begin{subfigure}{0.4\textwidth}
        \centering
        \begin{tikzpicture}[scale=0.86]
            \draw (0, 1.61803+1) -- (0, 1.61803) -- (0,0) -- (1.61803, 1.61803+1);
            \draw (1.61803, 0) -- (1.61803+1, 1.61803); 
            \draw (1.61803+1, 0) -- (1.61803+1, 1.61803) -- (1.61803, 0); 
            \draw (1.61803, 1.61803) -- (1.61803, 1.61803+1) -- (0, 0); 
            \draw[red] (1.61803, 0) -- (1.61803+1, 1.61803) -- (1.61803, 1.61803) -- (0, 0);
            \draw[red] (0, 1.61803+1) -- (0, 1.61803) -- (1.61803, 1.61803+1);
            \filldraw[red] (0, 1.61803+0.5) circle (2pt) node[anchor=west] {1};
            \filldraw[red] (1.61803/2, 0.5*1.61803) circle (2pt) node[anchor=west] {2};
            \filldraw[red] (1.61803/2, 1.61803+0.5) circle (2pt) node[anchor=west] {3};
            \filldraw[red] (1.61803+0.5, 0) circle (2pt) node[anchor=north] {4};
            \filldraw[red] (1.61803+0.5, 0.5*1.61803) circle (2pt) node[anchor=south] {5};
            \draw (1.61803, 1.61803) -- (1.61803+1, 1.61803);
            \draw (0, 0) -- (1.61803+1, 0);
            \draw (0, 1.61803+1) -- (1.61803, 1.61803+1);
        \end{tikzpicture}  
    \end{subfigure}
    \caption{Image under $\sigma_3$ before and after cut and paste.}
    \label{sigma3app}
\end{figure}

\begin{lemma} \label{prop:permutation}
For each $\sigma_i$, the permutation of the Weierstrass points (as labeled in Figure \ref{fig:goldLWP}) is a product of disjoint transpositions in the symmetric group $S_5$ given by the following:
\begin{center}
    $\tau_0 = (1\; 2)(3\; 4)$, $\tau_1 = (1\; 3)(2\; 5)$, $\tau_2 = (1\; 4)(3\; 5)$, $\tau_3 = (2\; 3)(4\; 5)$
\end{center}
with $\tau_i$ being the permutation associated with $\sigma_i$.
\end{lemma}

\begin{proof}
    This can be verified by the image of the golden L under each $\sigma_i$ after cut and paste, as shown in Figures \ref{sigma0app} - \ref{sigma3app}.
\end{proof}

\begin{remark}
Since $\tau_i$ is a product of disjoint transpositions, it has order 2, so $\tau_i^{-1} = \tau_i$ for all i. As a result, the permutation associated with $\sigma_i^{-1}$ is $\tau_i$.
\label{rem:order}
\end{remark}

\subsection{Long and short cylinders}
Davis and Leli\`evre in \cite{dl} show that both the golden L and double pentagon decompose into exactly $2$ cylinders in any given periodic direction.  In particular, the decomposition gives a \textit{long cylinder} and a \textit{short cylinder}, and the ratios of the circumferences of these two cylinders is always $\phi$ (see \cite[\S 2.1]{dl}).

\section{Statement of algorithm and proof}\label{sec2}
The purpose of this section is to address the following question stated by McMullen \cite{mcmullenQ}:
\begin{question}
Given an edge midpoint and periodic direction on the regular pentagon billiard table, will the trajectory be long, short, or a saddle connection?
\end{question}

We find that we can define an explicit procedure for determining which midpoints will give a short or long trajectory for any periodic direction.  The procedure also shows that starting from one of the midpoints will give a section of a saddle connection. 

\begin{algo} \label{alg:MainAlgorithm}
Given a vector $\vec{v}$ in a periodic direction in the regular pentagon, a starting midpoint on side $J \in (1,2,3,4,5)$, as labeled in Figure \ref{fig:goldLWP}, we can determine if the trajectory is long, short, or a saddle connection through the following steps: 
\begin{enumerate}
    \item Find the tree word $a = k_1 k_2 \dots k_n$ associated with $P^{-1}\vec{v}$.
    \item Compute the product $$\tau := \tau_{k_1} \cdot \tau_{k_{2}} \cdots \tau_{{k_n}} \in S_5.$$
    \item Then,
    \begin{itemize}
        \item If $\tau(J) \in \{1, 2\}$, then the trajectory is short.
        \item If $\tau(J) \in \{3, 4\}$, then the trajectory is long.
        \item If $\tau(J) = 5$, then the trajectory is a saddle connection.
    \end{itemize}  
\end{enumerate}
\end{algo}

\begin{proof}
    Corollary \ref{cor:PinvIsPer} tells us $P^{-1}\vec{v}$ is periodic in the golden L, and Theorem \ref{thm:PerDirec} gives us a corresponding tree word $k_1 k_2 \cdots k_n$ such that $P^{-1}\vec{v} = \ell_{\vec{v}}\sigma_{k_n}\cdots \sigma_{k_2} \sigma_{k_1} \big(\begin{smallmatrix} 1\\ 0 \end{smallmatrix}\big)$ for some length $\ell_{\vec{v}}$. We then apply the following sequence of inverse matrices to both the golden L and the direction vector $P^{-1}\vec{v}$:
    $$\sigma_{k_1}^{-1} \sigma_{k_2}^{-1} \cdots  \sigma_{k_n}^{-1}.$$
    There's two things that happen. First, the direction vector is changed. All the $\sigma_i$'s cancel out, leaving us with a horizontal vector $\ell_{\vec{v}}\big(\begin{smallmatrix} 1\\ 0 \end{smallmatrix}\big)$.
    
    Secondly, the golden L is acted on by this product. Recall $\sigma_i$ is in the Veech group of the golden L, so $\sigma_i^{-1}$ must also be in the Veech group. Then by closure, the product above must also be in the Veech group. As a result, we may perform cut and paste operations and get back to the golden L. Moreover, by Proposition \ref{prop:permutation} and Remark \ref{rem:order}, the labeled midpoints (the Weierstrass points) will be permuted by the product $\tau := \tau_{k_1} \cdot \tau_{k_{2}} \cdots \tau_{{k_n}}$. Thus, midpoint $J$ will be sent to midpoint $\tau(J)$. We then have a horizontal direction originating from midpoint $\tau(J)$ on the golden L. Figure \ref{horizL} shows us clearly if midpoint $\tau(J)$ is in a long cylinder, short cylinder, or on a saddle connection in the horizontal direction. Finally, we use Lemma \ref{lem:LongShortConnec} to jump back to regular pentagon. The result follows immediately.
\end{proof}

\begin{figure}[h]
    \centering
    \begin{tikzpicture}[scale=1.45]
        \draw (0,0) -- (1.61803, 0) -- (1.61803^2, 0) -- (1.61803^2, 1.61803) -- (1.61803, 1.61803) -- (1.61803, 1.61803^2) -- (0, 1.61803^2) -- (0, 1.61803) -- (0,0);
        \draw [red] (1.61803, 0) -- (1.61803^2, 0) -- (1.61803, 1.61803) -- (0, 1.61803^2) -- (0, 1.61803) -- (1.61803, 0);
        \filldraw[red] (0, 1.61803+0.5) circle (2pt) node[anchor=south east] {1};
        \filldraw[red] (1.61803/2, 1.61803+0.5) circle (2pt) node[anchor=south west] {2};
        \filldraw[red] (1.61803/2, 1.61803/2) circle (2pt) node[anchor=south west] {3};
        \filldraw[red] (1.61803+0.5, 1.61803/2) circle (2pt) node[anchor=south west] {4};
        \filldraw[red] (1.61803+0.5, 0) circle (2pt) node[anchor=south] {5};
        \draw [ultra thick, <->] (0, 1.61803/2) -- (1.61803^2, 1.61803/2);
        \draw [ultra thick, <->] (0, 1.61803 + 1/2) -- (1.61803, 1.61803 + 1/2);
        \draw [ultra thick, <->] (1.61803, 0) -- (1.61803^2, 0);
    \end{tikzpicture}  
    \caption{A short trajectory (top), long trajectory (middle), and saddle connection (bottom) in the horizontal direction in the golden L.}
    \label{horizL}
\end{figure}
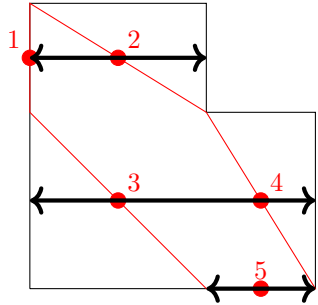

In the following Example, we illustrate the algorithm using the tree word $21$:

\begin{example}
	Consider the direction $21$ in the regular pentagon. The algorithm tells us to evaluate
    $$\tau = \tau_2 \cdot \tau_1 =(1\; 4)(3\; 5) \cdot (1\; 3)(2\; 5) = (1\; 5 \; 2\; 3\; 4).$$
    Then $\tau(1) = 5$, $\tau(2) = 3$, $\tau(3)=4$, $\tau(4)=1$, and $\tau(5)=2$. Thus, the algorithm tells us that midpoints 4 and 5 are in the short cylinder, midpoints 2 and 3 are in the long cylinder, and midpoint 1 is a saddle connection. To check this, we can draw in exactly what the trajectories look like in the regular pentagon. We first find the direction vector $\vec{v}$ in the golden L that corresponds to the tree word $21$ (see Example \ref{ex:treeword}). That turns out to be $\vec{v} = (2\phi + 2, 2 + 1)$. We then calculate the corresponding direction on the regular pentagon:
    $$P\vec{v} = \begin{pmatrix} 1 & \cos{\pi/5} \\ 0 & \sin{\pi/5} \end{pmatrix} \cdot \begin{pmatrix} 2\phi + 2 \\ 2 + 1 \end{pmatrix} \approx \begin{pmatrix} 8.66 \\ 2.49 \end{pmatrix}.$$
    We can then draw in the trajectory using the standard rules of billiards. The long and short trajectories can be seen in Figure \ref{fig:LongShortEx}, along with the midpoints the trajectories hit. 
\end{example}

\begin{figure}
    \centering
    \begin{subfigure}{0.45\textwidth}
        \centering
        \includegraphics[width=0.95\textwidth]{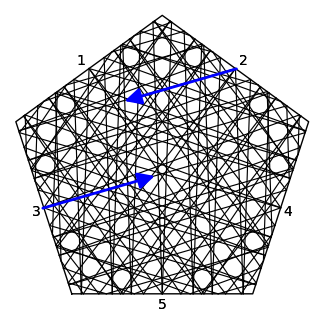}
    \end{subfigure}
    \begin{subfigure}{0.45\textwidth}
    \centering
        \includegraphics[width=0.95\textwidth]{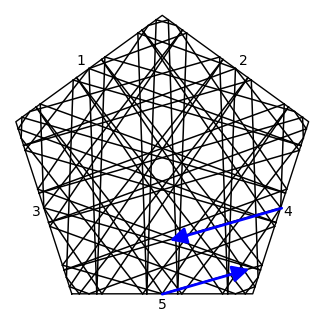}
    \end{subfigure}
    \caption{A long trajectory (left) and a short trajectory (right) in the direction $21$.}
    \label{fig:LongShortEx}
\end{figure}

This problem may be extended in a number of interesting directions.  One interesting question is the following.
\begin{question}
    How does this algorithm generalize to regular $(2n+1)$-gons?
\end{question}
We remark that this method only readily generalizes to surfaces with periodic points under the Veech group that are exactly the Weierstrass points of the surface, which are also the midpoints of edges of the regular polygon.

\section{Simplifying Tree Words}\label{sec3}
For a given finite tree word $a = k_1k_2 \dots k_n$, we define the \textit{derived word} $a'$ to be $a$ with the removal of repeated pairs of numbers. That is, if $k_i=k_{i+1}$, then both $k_i$ and $k_{i+1}$ will be removed. We say $a$ is a \textit{base word} if $a=a'$. Thus, every tree word $a$ has a base word, and that base word can be reached by deriving $a$ until there are no more pairs of numbers.

Let $\Omega$ be the space of all possible tree words. Let $\langle e \rangle$ denote the empty word. Define $\mu : \Omega \to \Omega$ to be the function that takes a tree word to its base word. 

\begin{example}
    Consider the tree word $a=231221$. Then $a'=2311$. We may derive again and get $a''=23$. Notice $a''$ does not have any pairs of numbers, so it cannot be derived further. Thus, the base word for $a=231221$ is $a''=23$. Furthermore, we write $\mu(231221)=23$.
\end{example}

We now state a corollary to our main theorem that may help reduce the calculations needed for determining the long and short cylinder decomposition for long tree words.
\begin{corollary}
    For any given tree word $a$, the corresponding base word $\mu(a)$ and $a$ have the same midpoints in the long cylinder, short cylinder, and on a saddle connection.
\end{corollary}

\begin{proof}
    Consider a tree word with a pair of numbers $k_i = k_{i+1}$. By our algorithm, we then consider the product $\tau= \cdots \tau_{k_i}\tau_{k_{i+1}} \cdots $. But by Remark \ref{rem:order}, $\tau_{k_i}\tau_{k_{i+1}}=(1)$, the identity element of $S_5$. Thus, this pair of numbers in our tree word cancel out, resulting in the derived tree word having the same midpoints in the same cylinders as our original tree word. This may then be repeated until we reach the base word, proving the corollary.
\end{proof}

This has an interesting consequence. That is, when we run the algorithm on a tree word, we may first reduce it down to its base word and then run the algorithm on that base word. This will the same result as if we ran it on the original word. As a result, for any tree word whose base word is $\langle e \rangle$, it is immediate that the midpoints $1,2$ lie in the short cylinder, midpoints $3,4$ lie in the long cylinder, and midpoint $5$ lies on a saddle connection. Thus, we ask the following question:

\begin{question}
    Given a tree word of length $2n$, what is the probability the base word will be the empty word $\langle e \rangle$, i.e., what is the probability that $\mu(a)=\langle e \rangle$ for a given tree word $a$ of length $2n$?
\end{question}

\subsection*{Acknowledgements}
We would like to thank Diana Davis, Samuel Leli\`evre, Jane Wang, and Sunrose Shrestha for helping us work on this problem in the Summer@ICERM 2021 REU. The idea for this algorithm was presented to us by Leli\`evre, who continually helped us understand and explore this problem.

\printbibliography
\end{document}